\documentclass[12pt]{amsart}
\usepackage{amssymb,latexsym}
\usepackage{amsmath}               
\usepackage{amsfonts}       
\usepackage{amsthm}                
\usepackage{setspace}
\usepackage{amssymb}
\usepackage{mathtools}

\newtheorem{theorem}{Theorem}[section]
\newtheorem{lemma}[theorem]{Lemma}
\newtheorem{proposition}[theorem]{Proposition}

\def\C{\mathbb{C}}
\def\N{\mathbb{N}}
\def\Z{\mathbb{Z}}
\def\Q{\mathbb{Q}}

\begin{document}

\baselineskip=17pt

\title[Diophantine equations with Euler polynomials]{Diophantine equations with Euler polynomials}

\author[D. Kreso]{Dijana Kreso}
\address{Institut f\"ur Analysis und Computational Number Theory
(Math A)\\ Technische Universit\"at Graz\\ Steyrergasse 30/II\\
8010 Graz, Austria}
\email{kreso@math.tugraz.at}

\author[Cs. Rakaczki]{Csaba Rakaczki}
\address{Institute of Mathematics\\ University of Miskolc\\
H-3515 Miskolc Campus, Hungary}
\email{matrcs@uni-miskolc.hu}

\date{}

\begin{abstract}
In this paper we determine possible decompositions of Euler polynomials $E_k(x)$, i.e. possible ways of writing Euler polynomials as a functional composition of polynomials of lower degree. Using this result together with the well-known criterion of Bilu and Tichy,
we prove that the Diophantine equation
$$-1^k +2 ^k - \cdots + (-1)^{x} x^k=g(y),$$
with $\deg(g) \geq 2$ and $k\geq 7$, has only finitely many integers solutions $x, y$ unless polynomial $g$ can be decomposed in ways that we list explicitly.
\end{abstract}



\maketitle

\section{Introduction}
If $K$ is a field and $g(x), h(x)\in K[x]$, then $f=g\circ h$ is a functional composition of $g$ and $h$ and $(g, h)$ is a (functional) {\it decomposition} of $f$ (over $K$). The decomposition is {\it nontrivial} if $g$ and $h$ are of degree at least $2$. A polynomial is said to be {\it indecomposable} if it is of degree at least $2$ and does not have a nontrivial decomposition. Given $f(x)\in K[x]$ with $\deg f>1$, a {\it complete decomposition} of $f$ is a decomposition $f=f_1\circ f_2 \cdots \circ f_m$, where polynomials $f_i\in K[x]$ are indecomposable for all $i=1, 2, \ldots, m$.  Two decompositions $f=g_1\circ h_1=g_2\circ h_2$ are said to be {\it equivalent} over $K$ if there exists a linear polynomial $\ell \in K[x]$ such that $g_2=g_1\circ \ell$ and $h_1=\ell \circ h_2$. Complete decomposition of a polynomial of degree greater than $1$ clearly always exists, but it does not need to be unique. In 1922, J. F. Ritt \cite{R22} proved that any two complete decomposition of $f\in \C[x]$ consist of the same number of indecomposable polynomials and moreover that the sequence of degrees of  polynomials in a complete decomposition of $f$ is uniquely determined by $f$, up to permutation. This result is known in literature as Ritt's first theorem. For more on the topic of polynomial decomposition we refer to \cite{S04}.

Ritt's polynomial decomposition results have been applied to a variety of topics. One such topic is the classification of polynomials $f$ and $g$ with rational coefficients such that the equation $f(x)=g(y)$ has infinitely many integer solutions. In 2000, Bilu and Tichy \cite{BT00} presented a complete and definite answer to this question. In the past decade the theorem of Bilu and Tichy has been applied to various Diophantine equations. For example, in \cite{BBKPT02} it is shown that the equation 
$$1^m + 2^m + \cdots +x^m=1^n + 2^n + \cdots +y^{n}$$
has only finitely many integer solutions $x, y$, provided $m, n\geq 2$ and $m\neq n$. In \cite{R04} Rakaczki investigated the question of the finiteness of the number of integer solutions $x, y$ of the equation 
$$1^m + 2^m + \cdots +x^m=g(y)$$
with an arbitrary $g(x)\in \Q[x]$. We mention that the study of Diophantine equations involving power sums of consecutive integers has a long history, dating back to the work of Sch\"{a}ffer in 1956, see \cite{S56}. In the present paper we study a related problem.

The purpose of this paper is to characterize those $g\in \Q[x]$ for which the diophantine equation 
\begin{equation}\label{EQ}
-1^k +2 ^k - \cdots + (-1)^{x} x^k=g(y)
\end{equation}
has infinitely many integer solutions. It is well known, see for instance \cite{AS72}, that the following relation holds:
\begin{equation*} \label{euler}
-1^k+2^k-3^k+\cdots +(-1)^n n^k=\frac{E_k(0)+(-1)^n E_k(n+1)}{2},
\end{equation*}
where $E_k(x)$ denotes the $k$-th Euler polynomial, which is defined by the following generating function:
$$\sum_{k=0}^{\infty}E_k(x)\frac{t^k}{k!}=\frac{2\exp(tx)}{\exp(t)+1}.$$
In the present paper we give a complete description of decompositions of Euler polynomials into polynomials with complex coefficients. Since Euler polynomials appear in many classical results and play an important role in various approximation and expansion formulas in discrete mathematics and in number theory (see for instance \cite{AS72}, \cite{B69}), we find that our result on Euler polynomials might be of broader interest.

\begin{theorem}\label{EP}
Euler polynomials $E_k(x)$ are indecomposable for all odd $k$.
If $k=2m$ is even, then every nontrivial decomposition of $E_k(x)$
over complex numbers is equivalent to
\begin{equation}\label{Euler}
E_k(x)=\widetilde{E}_m\left(\left(x-\frac{1}{2}\right)^{2}\right),\ 
\textnormal{where}\  \widetilde{E}_m(x)=\sum_{j=0}^{m}{2m\choose2j}\frac{E_{2j}}{4^{j}}x^{m-j}
\end{equation}
and $E_j$ is the $j$-th Euler number defined by $E_j=2^j E_j(1/2)$.
In particular, the polynomial $\widetilde{E}_m(x)$ is
indecomposable for any $m\in\N$.
\end{theorem}

Theorem \ref{EP} together with the aforementioned criterion of Bilu and Tichy enables us to prove the following theorem.

\begin{theorem}\label{THM}
Let $k\geq 7$ be an integer and $g(x)\in\Q[x]$ with
$\deg g \geq 2$. Then the Diophantine equation \eqref{EQ}
has only finitely many integer solutions unless we are in one of the following cases
\begin{itemize}
 \item[ i)]
$g(x)=f\left(E_k(p(x))\right)$,
 \item[ii)]
$g(x)=f\left(\widetilde{E}_s\left(p(x)^2\right)\right)$,
 \item[iii)]
$g(x)=f\left(\widetilde{E}_{s}\left(\delta(x)p(x)^{2}\right)\right)$,
\item[iv)]
$g(x)=f\left(\widetilde{E}_{s}\left(\gamma\delta(x)^{t}\right)\right)$,
\item[v)]
$g(x)=f\left(\widetilde{E}_{s}\left(\left(a\delta(x)^{2}+b\right)p(x)^{2}\right)\right)$,
\end{itemize}
where $a,b,\gamma \in \Q\setminus\{0\}$, $t\geq 3$ odd, $E_k(x)$ is the $k$-th Euler polynomial,  $p(x)\in \Q[x]$, $\delta(x)\in \Q[x]$ is a linear polynomial,
$$f(x)=\pm \frac{x}{2} +\frac{E_k(0)}{2} \quad \textnormal{and} \quad \widetilde{E}_s(x)=\sum_{j=0}^{s}{2s\choose
2j}\frac{E_{2j}}{4^j}x^{s-j}.$$
\end{theorem}

The proof of Theorem of Bilu and Tichy relies on Siegel's classical theorem on integral points on curves, which is ineffective. Consequently, the Theorem \ref{THM} is ineffective. 

In the proof of Theorem \ref{THM} in each of the exceptional cases, we find an infinite family of integer solutions of the equation \eqref{EQ}.

In relation to our problem we mention a paper by Dilcher \cite{D86}, where the effective finiteness
theorem is established for the diophantine equation
\begin{equation}\label{dilchereq}
-1^k +3^k - \cdots - (4x-3)^k + (4x-1)^k=y^n,
\end{equation}
which was viewed as a "character-twisted" analogue of Sch\"{a}ffer's
equation \cite{S56}, and a recent paper by Bennett
\cite{B11}, where the same equation
was completely solved for $3\leq k \leq 6$ using
methods from Diophantine approximations, as well as techniques
based upon the modularity of Galois representations.
Using our techniques, one can obtain ineffective finiteness theorems
of a similar flavor as Theorem \ref{THM} for the diophantine equation
\begin{equation}\label{gdilchereq}
-1^k +3^k - \cdots - (4x-3)^k + (4x-1)^k=g(y),
\end{equation}
with $k\in \N$ and an arbitrary $g(x)\in \Q[x]$.

\section{Decomposition of Euler polynomials}\label{Sec2}

In this section we recall and establish some results on polynomial decomposition and then use them to determine decomposition properties of Euler polynomials. 

The following lemma describes the structure of the set of all decompositions of a fixed monic polynomial into two decomposition factors in the case when the corresponding field is either of characteristic $0$ or of positive characteristic, but the degree of the polynomial is not divisible by the characteristic of the field. This case is known in literature as the {\it tame case}. In the tame case, there are known analogues of Ritt's theorems. The case in which the degree of the polynomial is divisible by the characteristic of the field is called {\it wild} and in this case analogues of Ritt's results do not hold, see \cite{DW74}. Similarly, the following lemma also fails in wild case.

\begin{lemma}\label{LM}
Let $f(x)\in K[x]$ be a monic polynomial such that $\textnormal{char}(K)\nmid\deg f$. Let $L$ be an arbitrary extension field of $K$. Then for any nontrivial decomposition $f=f_1\circ f_2$ with $f_1(x), f_2(x)\in L[x]$, there exists a unique decomposition
$f=\tilde f_1\circ \tilde f_2$, such that the following conditions are satisfied:
\begin{itemize}
\item[i)] $\tilde f_1(x)$ and  $\tilde f_2(x)$ are monic polynomials with coefficients in $K$,
\item[ii)] $\tilde f_1\circ  \tilde f_2$ and $f_1\circ f_2$ are equivalent over $L$,
\item[iii)] if we denote $t:=\deg \tilde f_1$, then the coefficient of $x^{t-1}$ in $\tilde f_1(X)$ equals $0$.
\end{itemize}

\end{lemma}
\begin{proof}
Let $f(x)=f_1(f_2(x))$ be a nontrivial decomposition of $f(x)\in K[x]$ with $f_1(x), f_2(x)\in L[x]$. Let $t=\deg f_1$, $k=\deg f_2$ and let $b_k\in L$ be the leading coefficient of $f_2(x)$. Then 
$$\hat f_1(x):=f_1(b_kx), \quad \hat f_2(x):=b_k^{-1}f_2(x)$$
 are clearly monic polynomials. Let $\hat a_{t-1}$ be the coefficient of $x^{t -1}$ in $\hat f_1(x)$. Let 
$$\tilde f_1(x):=\hat f_1(x-t^{-1 }\hat a_{t-1}), \quad \tilde f_2(x):=\hat f_2(x)+t^{-1}\hat a_{t-1}.$$
It is easy to verify that the coefficient of $x^{t-1}$ in $\tilde f_1(x)$ is $0$ and since $\hat f_1$ and $\hat f_2$ are monic, so are $\tilde f_1$ and  $\tilde f_2$.  Let $\tilde f_1(x)=x^{t}+a_{t-1}x^{t-1}+\cdots+a_0$ and $\tilde f_2(x)=x^{k}+b_{k-1}x^{k-1}+\cdots +b_0$, where $a_i, b_j\in L$, for $i=0, 1, \ldots, t$, $j=0, 1, \ldots, k$, and $a_{t-1}=0$. Further let $f(x)=c_nx^n+\cdots+c_1x+c_0$. Now we can easily see that $\tilde f_1$ and $\tilde f_2$ are uniquely determined  and have coefficients in $K$. From
\begin{equation}\label{EXP}
f(x)=\tilde f_1(\tilde f_2(x))=\tilde f_2(x)^t+a_{t-2} \tilde f_2(x)^{t-2}+ \cdots +a_1\tilde f_2(x)+a_0,
\end{equation}
by expanding $\tilde f_2(x)^t$ we get the following system of equations which
completely determine coefficients of $\tilde f_2(x)$:
\begin{equation}\label{SE}
\begin{cases}
 c_{n-1}=t b_{k-1} \\ c_{n-2}=tb_{k-2}+{t \choose 2}b_{k-1}^2 \\ \quad \quad \ \vdots \\ c_{n-k}=tb_{0}+\sum\limits_{i_1+2i_2+\cdots+(k-1)i_{k-1}=k}^{}
 d_{i_1, i_2, \ldots, i_{k-1}}\ b_{k-1}^{i_1}b_{k-2}^{i_2}\ldots b_{1}^{i_{k-1}},
\end{cases}
\end{equation}
where
$$d_{i_1, i_2, \ldots, i_{k-1}}={t \choose {i_1, i_2, \ldots, i_{k-1}}}.$$
Since $c_{i}\in K$, it follows that $b_{i}\in K$ for all $i=0, 1, \ldots,  k-1$ and hence $\tilde f_2(x) \in K[x]$. Furthermore, from \eqref{EXP} it follows that the coefficients of $\tilde f_1$ are uniquely determined by $F$ and $\tilde f_2$. Recursively, $a_i \in K$ for all $i=t-2, \ldots, 1, 0$. Hence, $\tilde f_1(x) \in K[x]$ as well.
\end{proof}

The proof of Lemma \ref{LM} fails when the degree of the polynomial is divisible by the characteristic of the field, since in this case there does not exist the multiplicative inverse of the degree of the polynomial in the field.

Lemma \ref{LM} implies that if $f\in K[x]$ is indecomposable over $K$, then it is indecomposable over any extension field of $K$. This result is well known. In fact, we built up a proof of Lemma \ref{LM} based on \cite[Theorem 6, Chapter 1.3]{S04}.

We will further need the following lemma.

\begin{lemma}\label{EQUIV}
Let $f\in K[x]$ such that $\deg f$ is not divisible by the characteristic of the field $K$. If $f=g_1\circ g_2=h_1\circ h_2$ and $\deg g_1 = \deg h_1$, and hence $\deg g_2=\deg h_2$, then there exists a linear polynomial $\ell \in K[x]$ such that $g_1(x)=h_1(\ell(x))$ and $h_2(x)=\ell(g_2(x))$.
\end{lemma}

\begin{proof}
The case $K=\C$ is contained already in \cite{R22}. Lemma was later proved in generality by Levi \cite {L42}. 
\end{proof}

The following observation will be of great help to the proof.

\begin{lemma}\label{DR}
Let $n$ be an even positive integer. If $$(x+1)^{n}-x^{n}=g(x)h(x)$$
with $g(x), h(x)\in\mathbb{R}[x]$, then the coefficients of $g(x)$
and $h(x)$ are either all positive or all negative.
\end{lemma}

\begin{proof}
We have $(x+1)^{n}-x^{n}=\prod_{i=1}^{n}(x+1-\omega_ix)$,
where $\omega_i=e^{\frac{2\pi i}{n}}$, $i=1, 2, \dots, n$. Let $n=2k$. Hence, $\omega_{2k}=1$,
$\omega_{k}=-1$, and $\omega_{2k-j}=\overline{\omega_j}$
for all $j=1,2,\ldots ,k-1$. Therefore we have
\begin{align*}\label{xx+1}
(x+1)^{n}-x^{n}&=(2x+1)\prod_{j=1}^{k-1}\left(x+1-\omega_jx\right)\left(x+1-\overline{\omega_j}x\right)\\
\nonumber &=(2x+1)\prod_{j=1}^{k-1}\left((2-(\omega_j+\overline{\omega_j}))x^{2}+(2-(\omega_j+\overline{\omega_j}))x+1\right).
\end{align*}
Clearly $2-(\omega_j+\overline{\omega_j})>0$ for all $j\in\left\{1,2,\ldots ,k-1\right\}$.
Now the assertion follows from the fact that the ring $\mathbb{R}[x]$ is a unique factorization domain.
\end{proof}

We will make an extensive use of the following theorem of Rakaczki \cite{R11}.
\begin{theorem}\label{RAK}
Let $m\geq 7$ be an integer. Then the polynomial
$E_m(x)+b$ has at least three simple zeros for arbitrary
complex number $b$.
\end{theorem}

Finally, to the proof of Theorem \ref{EP} we need the following proposition, in which we collect some well known properties of Euler polynomials,  which will be used in the sequel, sometimes without particular reference, see \cite{B69} for proofs.

\begin{proposition}\label{PROP}
\mbox{}
\begin{itemize}
\item[i)]  $E_n(x)=(-1)^{n}E_n(1-x)$;
\item[ii)]  $E_n(x+1)+E_n(x)=2x^{n}$;
\item[iii)]  $E'_n(x)=nE_{n-1}(x)$;
\item[iv)] $E_5(x)$ is the only Euler polynomial with a multiple root.
\item[v)] If $E_k$ denotes the $k$-th Euler number, which is defined by $E_k=2^k E_k(1/2)$, then
$$E_n(x)=\sum_{k=0}^{n}{n\choose k}\frac{E_k}{2^{k}}\left(x-\frac{1}{2}\right)^{n-k},$$ 
i.e.\@ $E_n(x)=\sum_{k=0}^{n}c_kx^k$ with
$$c_k=\sum_{j=0}^{n-k}{n \choose j}\frac{E_j}{2^j}{n-j \choose k}\left(\frac{-1}{2}\right)^{n-k-j},$$
for $k=0,1, \ldots, n$. In particular,
$$c_n=1, \ c_{n-1}=-\frac{1}{2}n, \ c_{n-2}=0, \ c_{n-3}=\frac{1}{4}{n \choose 3}, \ \textnormal{etc.} $$
\end{itemize}
\end{proposition}

\begin{proof}[Proof of Theorem \ref{EP}]
Let $n\in \N$ and
\begin{equation}\label{decomposition}
E_n(x)=g(h(x))
\end{equation}
be a nontrivial decomposition of the $n$-th Euler polynomial.
By Lemma \ref{LM} we may assume that polynomials
$g(x)$ and $h(x)$ are monic with rational coefficients; let $g(x)=x^{t}+a_{t-1}x^{t-1}+\cdots
+a_0\in \Q[x]$ and $h(x)=x^{k}+b_{k-1}x^{k-1}+\cdots +b_0\in\Q[x]$. By the same lemma we may assume $a_{t-1}=0$. Note  $t, k\geq 2$ by assumption. Using \eqref{SE} we can express the coefficients of $h(x)$ in terms of coefficients of the $E_n(x)$, which are given in Proposition \ref{PROP}, so
\begin{align}\label{B'S}
b_{k-1}&=-\frac{k}{2}, \quad  b_{k-2}=-\frac{(t-1)k^{2}}{8}, \\  
\nonumber{}b_{k-3}&=\frac{1}{4}{k\choose 3}+\frac{(t-1)k^{2}(k-2)}{16}.
\end{align}
From $E_n(1-x)=(-1)^nE_n(x)$, it follows that 
\begin{equation}\label{oddeven}
g(h(1-x))=(-1)^n g(h(x)).
\end{equation}
We first consider the case when $n$ is even. Then $g(h(1-x))=g(h(x))$. From Lemma \ref{EQUIV}, by using $a_{t-1}=0$, we get that 
either $h(1-x)=h(x)$ or $h(1-x)=-h(x)$ and $g(x)=g(-x)$.
In the former case $k$ is even. From Proposition \ref{PROP} we get
$$2\left((x+1)^n - x^n\right)=E_n(-x-1)-E_n(x)=g(h(-x-1))-g(h(x)),$$
so $(x+1)^{n}-x^{n}$  is divisible by $h(-x-1)-h(x)$ in $\Q[x]$. Note that the leading coefficient of
$h(-x-1)-h(x)$ is $k-2b_{k-1}=2k$. If $k\geq 4$, from Lemma \ref{DR} it follows that the coefficient of $x^{k-4}$ in $h(-x-1)-h(x)$ is positive, so
\begin{equation}\label{coeff}
{k\choose 4}-{k-1\choose 3}b_{k-1}+{k-2\choose
2}b_{k-2}-{k-3\choose 1}b_{k-3}>0.
\end{equation}
Using \eqref{B'S} we obtain
$${k\choose 4}> \frac{(t-1)k^{2}(k-2)(k-3)}{16},$$
wherefrom $t\leq 1$, contradicting the assumption. Since $k$ is even, we conclude $k=2$ and hence $n=2t$. Lemma \ref{EQUIV} implies that this decomposition is equivalent to the decomposition \eqref{Euler}. 

In the case when $h(1-x)=-h(x)$ and $g(x)=g(-x)$ one can deduce that $k$ is odd, $t$ is even, $g(x)=x^{t}+a_{t-2}x^{t-2}+\cdots +a_2x^{2}+a_0$ and 
$$E_n(x)=g(h(x))=g_1(h_1(x)),$$
where 
$$g_1(x)=x^{t/2}+a_{t-2}x^{t/2-1}+\cdots +a_2x+a_0, \quad h_1(x)=h(x)^{2}.$$
But then $h_1(x)=h_1(1-x)$ and we can use the argument above to get a contradiction provided $t\geq 4$. If $t=2$, then $g(x)=x^{2}+a_0$ and hence
$E_n(x)=h(x)^{2}+a_0.$ From Theorem \ref{RAK} it follows that this is possible only when $n\leq6$. Since $n\geq 4$ and $k$ is odd, it follows that the only possibility is $n=6$, but a direct calculation shows that $E_6(x)$ is not of this form.

If $n$ is odd, then $k$ and $t$ are also odd. Proposition \ref{PROP}  implies
$$2x^{n}=E_n(x)-E_n(-x)=g(h(x))-g(h(-x)),$$
wherefrom we deduce that $h(x)-h(-x)$ divides $2x^{n}$ in $\Q[x]$. Hence, $h(x)-h(-x)=qx^l$ with $q\in \Q$ and $l\leq n$. By expanding $h(x)-h(-x)$ we obtain $l=k$, $q=2$ and $b_{k-2}=0$, which together with \eqref{B'S}  implies $t=1$ or $k=0$, contradicting the assumption $k,t \geq 2$. Hence, Euler polynomials with odd index are indecomposable.
\end{proof}

\section{Application of the Theorem of Bilu and Tichy}\label{Sec3}

To the proof of Theorem \ref{THM} we need some auxiliary results.
First we recall the finiteness criterion of Bilu and Tichy \cite{BT00}. 

We say that the equation $f(x)=g(y)$ has infinitely many {\it rational solutions with a
bounded denominator} if there exists a positive integer $\lambda$ such that
$f(x)=g(y)$ has infinitely many rational solutions $x, y$
satisfying $\lambda x, \lambda y \in \Z$. If the equation $f(x)=g(y)$ has only finitely many rational solutions with a bounded denominator, then it clearly has only finitely many integer solutions.

We further need to define five kinds of so-called {\it standard pairs} of polynomials.

In what follows $a$ and $b$ are nonzero rational numbers, $m$ and $n$ are positive integers, $r \geq 0$ is an integer and $p(x) \in \Q [x]$ is a nonzero polynomial (which may be constant). 

A \textit{standard pair over $\Q$ of the first kind}  is $\left(x^m, a x^{r} p(x)^m\right)$, or switched, i.e\@ $\left(a x^{r} p(x)^m, x^m\right)$, where $0 \leq r < m$, $\gcd(r,m)=1$ and $r + \deg p > 0$.

A \textit{standard pair over $\Q$ of the second kind} is $\left(x^2, \left(a x^2 + b\right) p(x)^2\right)$, or switched.

Denote by $D_{m} (x,a)$ the $m$-th Dickson polynomial with parameter $a$, defined by the functional equation 
$$D_{m}\left (z+\frac{a}{z}, a\right) = z^{m} + \left(\frac{a}{z}\right)^{m}$$
or by the explicit formula
\begin{equation}\label{dickson}
D_{m} (x,a) = \sum_{i=0}^{\left\lfloor m / 2 \right\rfloor}{\frac{m}{m - i} \binom{m - i}{i} (-a)^i x^{m - 2i}}.
\end{equation}

A \textit{standard pair over $\Q$ of the third kind} is $\left(D_{m} \left(x,a^{n}\right), D_{n} \left(x,a^{m}\right)\right)$, where $\gcd (m, n) = 1$.

A \textit{standard pair over $\Q$ of the fourth kind} is
$$\left(a^{-m/2} D_{m} (x,a), -b^{-n/2} D_{n} (x,b)\right),$$
where $\gcd (m,n) = 2$.

A \textit{standard pair over $\Q$ of the fifth kind} is $\left(\left(a x^2 - 1\right)^3, 3x^4 - 4x^3\right)$, or switched.

\begin{theorem}\label{BT}
Let $f(x)$ and $g(x)$ be non-constant polynomials in $\Q[x]$.
Then the following assertions are equivalent.
\begin{itemize}
\item[-] The equation $f(x)=g(y)$ has infinitely many
rational solutions with a bounded denominator;
\item[-] We have 
$$f(x)=\varphi\left(f_{1}\left(\lambda(x)\right)\right), \quad g(x)=\varphi\left(g_{1}\left(\mu(x)\right)\right),$$ 
where $\lambda(x)$ and $\mu(x)$ are linear polynomials in $\Q[x]$,
$\varphi(x)\in\Q[x]$, and $\left(f_{1}(x),g_{1}(x)\right)$ is a
standard pair over $\Q$ such that the equation $f_1(x)=g_1(y)$
has infinitely many rational solutions with a bounded denominator.
\end{itemize}
\end{theorem}

The following theorem for hyperelliptic equations is due to Baker \cite{BA69}.

\begin{theorem}\label{Baker}
Let $f(x)\in\Q[x]$ be a polynomial with at least three
simple roots. Then all the integer solutions of the equation $f(x)=y^{2}$ satisfy $\max\left\{|x|,|y|\right\}\leq C$,
where $C$ is an effectively computable constant that depends only on
the coefficients of $f$.
\end{theorem}

For $P(x)\in\mathbb{C}[x]$, a complex number $c$ is said to be an
\textit{extremum} if $P(x)-c$ has multiple roots. If $P(x)-c$ has $s$ multiple roots,
the \textit{type} of $c$ is the tuple $(\alpha_{1},\alpha_{2},\ldots,\alpha_{s})$ 
of multiplicities of its roots in an increasing order.
Clearly $s < \deg P$, $(\alpha_{1},\alpha_{2},\ldots,\alpha_{s})\neq (1, 1, \ldots, 1)$ 
and $\alpha_1+\alpha_2 + \cdots +\alpha_s=\deg P$.

In what follows $D_k(x, a)$ denotes the Dickson polynomial of degree $k\in \N$ with parameter $a\in \Q\setminus\{0\}$. The following result on Dickson polynomials can be found in \cite[Proposition 3.3]{B99}.

\begin{theorem}\label{bdickson}
If $k\geq 3$, $D_{k}(x,a)$ has exactly two extrema and those are
$\pm 2a^{\frac{k}{2}}$. If $k$ is odd, then both are of type
(1,2,2,\ldots ,2). If $k$ is even, then $2a^{\frac{k}{2}}$ is of
type (1,1,2,\ldots ,2) and $-2a^{\frac{k}{2}}$ is of type
(2,2,\ldots ,2).
\end{theorem}

What follows is a technical lemma which will be
needed in the proof of Theorem \ref{THM}.

\begin{lemma}\label{edicksonepower}
The polynomial $E_n(cx+d)$ is neither of the form $ux^{q}+v$ with
$q\geq 3$, nor of the form $uD_k(x,a)+v$ with $k>4$, where $c, u\in\Q\setminus\left\{0\right\}$, $d, v\in\Q$.
\end{lemma}

\begin{proof}
Suppose $E_n(cx+d)=ux^{q}+v$ with $q\geq 3$, so $q=n$. It follows that the
polynomial $(E_n(cx+d)-v)'=ncE_{n-1}(cx+d)$ has a zero with multiplicity $n-1$. This is not possible, see Proposition \ref{PROP}.
Now assume that $E_n(cx+d)=uD_k(x,a)+v$ and $n\geq 7$. So, $k=n$ and
$$D_n(x,a)\pm 2a^{\frac{n}{2}}=\frac{1}{u}\left(E_n(cx+d)-v\pm 2ua^{\frac{n}{2}}\right).$$
Then from Theorem \ref{RAK} it follows that $D_n(x,a)\pm
2a^{\frac{n}{2}}$ has at least three simple zeros, which contradicts
Theorem \ref{bdickson}. In the case when $n=5$ and $n=6$, a direct
calculation shows that $E_n(cx+d)$ is not of the form
$uD_n(x,a)+v$. We remark that
$$E_4\left(cx+\frac{1}{2}\right)=c^{4}D_4\left(x,\frac{3}{8c^{2}}\right)+\frac{1}{32}.$$
\end{proof}

\begin{proof}[Proof of Theorem \ref{THM}]
We recall
$$-1^k +2 ^k + \cdots + (-1)^{n} n^k=\frac{E_k(0)+(-1)^n E_k(n+1)}{2}.$$
Therefore, the study of integer solutions of the equation \eqref{EQ} reduces to the study of solutions of the equations
\begin{align}
\frac{E_k(0)+E_k(2x+1)}{2}&=g(y) \label{eqnepar}\\
\frac{E_k(0)-E_k(2x)}{2}&=g(y). \label{eqpar}
\end{align}
in integers  $x, y$ with $x$ positive. We can study these two equations at once by writing
\begin{equation}\label{EQUATION}
f(E_k(h(x)))=g(y),
\end{equation}
where the equation \eqref{eqnepar} corresponds to polynomials
\begin{equation}\label{Case1}
f(x)=\frac{E_k(0)+x}{2}, \quad h(x)=2x+1, 
\end{equation}
and the equation \eqref{eqpar} to polynomials
\begin{equation}\label{Case2}
f(x)=\frac{E_k(0)-x}{2}, \quad h(x)=2x. 
\end{equation}
We further denote
\begin{equation}\label{composition}
F_k(x)=f(E_k(h(x))).
\end{equation}
If $\deg g=2$, then the equation \eqref{EQUATION} can be re-written as
$$df(E_k(h(x)))=ay^{2}+by+c$$
with $a, b, c, d\in\Z$, $a, d\neq 0$, which can be transformed into
\begin{equation}\label{g21}
uE_k(h(x))+v=(2ay+b)^{2},
\end{equation}
where $u, v\in \Q$ and $u\neq 0$. From Theorem
\ref{Baker} and Theorem \ref{RAK}, we get that the equation
\eqref{g21} has only finitely many integer solutions $x, y$,
which can be effectively determined, since $k\geq 7$ by assumption.

Let $\deg g>2$. Suppose that the equation
$\eqref{EQUATION}$ has infinitely many integer solutions.
By Theorem \ref{BT} there exists $\varphi(x)\in \Q[x]$,
linear polynomials $\lambda(x), \mu(x) \in \Q[x]$ and a
standard pair $(f_{1}(x),g_{1}(x))$ over $\Q$ such that
\begin{equation}\label{felbontas}
F_k(x)=\varphi(f_{1}(\lambda(x))), \quad
g(x)=\varphi(g_{1}(\mu(x))).
\end{equation}
Then from Theorem \ref{EP} and \eqref{composition} we get that either
 $\deg \varphi =k$ or $\deg \varphi=1$ or $\deg \varphi=k/2$.

\underline{Consider the case $\deg \varphi=k$}. Then from
\eqref{felbontas} we get $\deg f_{1}=1$, so
$F_k(x)=\varphi(t(x))$, where $t(x)\in\Q[x]$ is a linear polynomial. Then clearly
$$F_{k}\left(t^{\langle-1\rangle}(x)\right)=\varphi(x),$$
wher $t^{\langle-1\rangle}$ denotes the inverse of $t$ with respect to functional composition. Then from \eqref{felbontas} we get
$$g(x)=\varphi(g_{1}(\mu(x)))=F_{k}\left(t^{\langle-1\rangle}\left(g_{1}(\mu(x))\right)\right),$$
that is
\begin{equation}
g(x)=f(E_k(p(x)))
\end{equation}
with
$p(x)=h\left(t^{\langle-1\rangle}\left(g_{1}(\mu(x))\right)\right)\in\Q[x]$. In this particular case the equation \eqref{EQUATION} turns into
\begin{equation}\label{degfk}
f(E_k(h(x)))=f(E_k(p(y)).
\end{equation}
If we let $r(x)\in \Q[x]$ be an integer valued polynomial which attains only positive values and $p(x)=h(r(x))$, then the equation \eqref{degfk} clearly has infinitely many positive integer solutions.

\underline{Consider the case $\deg \varphi=1$.}
Let $\varphi(x)=\varphi_{1}x+\varphi_{0}$, where $\varphi_{1}, \varphi_{0}\in\Q$ and $\varphi_{1}\neq 0$. From \eqref{felbontas} it follows that
$$F_k\left(\lambda^{\langle-1\rangle}(x)\right)=\varphi(f_1(x))=\varphi_1f_1(x)+\varphi_0$$
and from \eqref{composition} it follows that
$$f\left(E_k\left(h\left(\lambda^{\langle-1\rangle}(x)\right)\right)\right)=F_k\left(\lambda^{\langle-1\rangle}(x)\right)=\varphi_1f_1(x)+\varphi_0.$$
Since $f(x), h(x), \lambda^{\langle-1\rangle}(x) \in \Q[x]$ are linear
polynomials, we have that
\begin{equation}\label{treci}
E_k(cx+d)=uf_1(x)+v
\end{equation}
for some $c, d, u, v\in \Q$, $c, u\neq 0$. Next we study five kinds of standard pairs of polynomials over $\Q$. 

First consider the case when $(f_{1}(x),g_{1}(x))$
 is a standard pair over $\Q$ of the first kind. From
\eqref{treci} we get that either $E_k(cx+d)=ux^{t}+v$ or $E_k(cx+d)=uax^{r}q(x)^{t}+v$,
where $0\leq r<t$, $(r,t)=1$ and $r+\deg q>0$.
In the former case we get a contradiction by Lemma
\ref{edicksonepower}, since by assumption $k=t\geq 7$. In the latter case, from Theorem \ref{RAK} we
get $t\leq 2$, contradiction.

Let now  $(f_{1}(x),g_{1}(x))$ be a
standard pair over $\Q$ of the second kind. Then either $E_{k}(cx+d)=ux^{2}+v$ or $E_{k}(cx+d)=u\left(ax^{2}+b\right)q(x)^{2}+v$.
The former case is not possible since $k\geq 7$ and the latter case is not possible by Theorem \ref{RAK}.

Next let $(f_{1}(x),g_{1}(x))$ be a standard pair of the
third or of the fourth kind. Then by \eqref{treci} it follows that
$$E_k(cx+d)=uD_k(x,w)+v,$$
where $w=a^{t}$ or $w=a$. Since $k\geq 7$ by assumption, we have a contradiction with Lemma \ref{edicksonepower} 

Finally, $(f_1(x), g_1(x))$ can not be a standard pair over $\Q$ of the fifth kind since $k\geq 7$.

\underline{Finally, consider the case $\deg \varphi=k/2$}. Then $k=2s$ and $\deg f_1 =2$. From \eqref{composition}
and \eqref{felbontas} we get
\begin{equation}\label{eqlambda}
E_k(x)=f^{\langle-1\rangle}\left(\varphi(f_1(\tau(x)))\right),
\end{equation}
where $\tau(x)$ is a linear polynomial in $\Q[x]$. Since  $\deg f_1=2$ and
$k\geq 7$, we have a nontrivial decomposition of $E_k(x)$ in \eqref{eqlambda}. 
From Theorem \ref{EP} it follows that there exists a linear
polynomial $u(x)$ such that
\begin{equation}\label{felbontas1}
\varphi(x)=f\left(\widetilde{E}_{s}\left(u(x)\right)\right), \quad
u\left(f_{1}(\tau(x))\right)=\left(x-\frac{1}{2}\right)^{2},
\end{equation}
which together with \eqref{felbontas} implies 
\begin{equation}\label{di}
g(x)=f\left(\widetilde{E}_{s}(u(g_1(\mu(x))))\right).
\end{equation}
Next we study five kinds of standard pairs over $\Q$. 

First consider the case when $\left(f_{1}(x),g_{1}(x)\right)$ is a
standard pair over $\Q$ of the first kind. If $f_1(x)=x^{t}$, then $t=2$ and hence $r=1$. Then $f_{1}(x)=x^2$ and $g_{1}(x)=axq(x)^{2}$ for some $q(x)\in \Q[x]$. Then from \eqref{felbontas1} we get $u(x)=x/c^{2}$ and hence from \eqref{di} it follows that
$$g(x)=f\left(\widetilde{E}_{s}\left(\frac{a\mu(x)q(\mu(x))^{2}}{c^{2}}\right)\right),$$
which we can write as
\begin{equation}\label{dec3}
g(x)=f\left(\widetilde{E}_{s}\left(\delta(x)p(x)^{2}\right)\right)
\end{equation}
with $\delta(x)=a\mu(x)/c^2$ and $p(x)=q(\mu(x)).$ Clearly $\delta(x), p(x)\in\Q[x]$ and $\deg\delta=1$. 
Now \eqref{EQUATION} turns into
\begin{equation}\label{elll}
f\left(\widetilde{E}_{s}\left(\left(h(x)-\frac{1}{2}\right)^{2}\right)\right)=
f\left(\widetilde{E}_{s}\left(\delta(y)p(y)^{2}\right)\right).
\end{equation}
We easily find a choice of parameters such that the equation \eqref{elll} has infinitely many positive integer solutions. For example, let $\delta(x)=x$, let $r(x)$ be a polynomial which attains positive odd integer values for every $x\in\N$ and let $p(x)=r(x)-1/2$. Either $h(x)=2x$ or $h(x)=2x+1$, see \eqref{Case1} and \eqref{Case2}, which corresponds to solutions
\begin{align*}
x&=\frac{(4k+3)r\left((4k+3)^2\right)-(2k+1)}{2},  \quad y=(4k+3)^{2},\\
\shortintertext{and}
x&=\frac{(4k+1)r\left((4k+1)^2\right)-(2k+1)}{2},  \quad y=(4k+3)^2,
\end{align*}
of the equation \eqref{elll} for any $k\in \N$, respectively.
Since $\deg f_1=2$, when $(f_{1}(x), g_{1}(x))=\left(ax^{r}q(x)^{t},x^{t}\right)$ with $0\leq r<t$, $(r, t)=1$, $r+\deg q>0$, then either
$r=0$, $t=1$ and $\deg q=2$ or $r=2$, $t\geq 3$ odd and $q(x)$ is a constant polynomial.
In the former case we have $g_{1}(x)=x$ and hence from \eqref{di} we get
$$g(x)=f\left(\widetilde{E}_{s}(u(g_1(\mu(x))))\right)=f\left(\widetilde{E}_{s}\left(\delta(x)p(x)^{2}\right)\right)$$
where $p(x)=1$ and $\delta(x)\in \Q[x]$ is a linear polynomial, which is a decomposition of $g$ that already appeared, see \eqref{dec3}. In the latter case we have $f_1(x)=bx^{2}$ and from \eqref{felbontas1} we get $u(x)=x/(bc^{2})$,
where $b\in\Q\setminus\{0\}$. Then
\begin{equation*}
\nonumber{}g(x)=f\left(\widetilde{E}_{s}\left(\frac{(\mu(x))^{t}}{bc^{2}}\right)\right),
\end{equation*}
which we can write as
\begin{equation} \label{2gx}
g(x)=f\left(\widetilde{E}_{s}\left(\gamma\delta(x)^{t}\right)\right),
\end{equation}
where $\gamma=1/(bc^{2})$, $\delta(x)=\mu(x)$. So, $\gamma\in \Q$, $\delta(x)$ is a linear polynomial in $\Q[x]$ and $t$ is odd. Now \eqref{EQUATION} turns into
\begin{equation}\label{eIV}
f\left(\widetilde{E}_{s}\left(\left(h(x)-\frac{1}{2}\right)^{2}\right)\right)=
f\left(\widetilde{E}_{s}\left(\gamma\delta(y)^{t}\right)\right).
\end{equation}
We easily find a choice of parameters such that the equation \eqref{eIV} has infinitely many integer solutions. 
For example, let $\gamma=1/4$, $\delta(x)=x$ and $t\geq 3$ odd. For $h(x)=2x$, and $h(x)=2x+1$, 
\begin{equation*}
x=\frac{(4k-1)^{t}+1}{4}, \quad y=(4k-1)^{2}
\end{equation*}
and
\begin{equation*}
x=\frac{(4k+1)^{t}-1}{4}, \quad y=(4k+1)^{2}
\end{equation*}
are solutions in positive integers of the equation \eqref{eIV}, respectively.

Next suppose that $(f_{1}(x),g_{1}(x))$ is a
standard pair of the second kind over $\Q$.  If $f_{1}(x)=(ax^{2}+b)q(x)^{2}$,
then $g_{1}(x)=x^{2}$, so from \eqref{di} we get
\begin{equation*}
g(x)=f\left(\widetilde{E}_{s}\left(u_{1}\mu(x)^{2}+u_{0}\right)\right),
\end{equation*}
which we can re-write as
\begin{equation} \label{gx21}
g(x)=f\left(\widetilde{E}_{s}\left(\left(a\delta(x)^{2}+b\right)p(x)^{2}\right)\right)
\end{equation}
with $a=u_1$, $b=u_0$, $\delta(x)=\mu(x)$ and $p(x)=1$. So, $p(x), \delta(x)\in \Q[x]$ and $\deg \delta=1$.
If $f_1(x)=x^{2}$, then from \eqref{felbontas1} we get $u(x)=x/c^{2}$ and hence
$$g(x)=f\left(\widetilde{E}_{s}\left(\frac{\left(a\mu(x)^{2}+b\right)q(\mu(x))^{2}}{c^{2}}\right)\right),$$
which we can re-write as
\begin{equation}\label{3gx}
g(x)=\left(\widetilde{E}_{s}\left(\left(a\delta(x)^{2}+b\right)p(x)^{2}\right)\right),
\end{equation}
with $p(x)=q(\mu(x))/c$ and $\delta(x)=\mu(x)$. Clearly $p(x), \delta(x)\in \Q[x]$ and $\deg \delta=1$. 
Then \eqref{EQUATION} turns into
\begin{equation}\label{eV}
f\left(\widetilde{E}_{s}\left(\left(h(x)-\frac{1}{2}\right)^{2}\right)\right)=
f\left(\widetilde{E}_{s}\left((a\delta(y)^{2}+b)p(y)^{2}\right)\right).
\end{equation}
Let $\delta(x)=x$, let $r(x)$ be any integer valued polynomial which attains only positive values and
$p(x)=4r(x)+1$. Let $a=1/2$ and $b=1/4$. Let $a_n$ and $b_n$ be such that
$$a_n+b_n\sqrt{2}=(3+2\sqrt{2})^{n}, \quad n\in\N.$$
For $h(x)=2x$, and $h(x)=2x+1$,
\begin{equation*}
x=\frac{a_{2n+1}(4r(y)+1)+1}{4}, \quad  y=b_{2n+1}
\end{equation*}
and
\begin{equation*}
x=\frac{a_{2n}(4r(y)+1)-1}{4},  \quad      y=b_{2n}, 
\end{equation*}
are solutions of the equation \eqref{eV}, respectively.
Let now $(f_{1}(x),g_{1}(x))$ in \eqref{felbontas} be a standard pair of the third kind over $\Q$.
Then 
$$(f_{1}(x),g_{1}(x))=\left(D_{2}\left(x,a^{t}\right),D_{t}\left(x,a^{2}\right)\right)$$
with odd $t$. Substituting
$f_{1}(x)=D_{2}(x,a^{t})=x^{2}-2a^{t}$ into $\eqref{felbontas1}$,
we get $u(x)=(x+2a^{t})/c^{2}$, so
\begin{equation}\label{di2}
g(x)=f\left(\widetilde{E}_{s}\left(\frac{D_{t}\left(\mu(x),a^{2}\right)+2a^{t}}{c^{2}}\right)\right).
\end{equation}
From Theorem \ref{bdickson} it follows that the polynomial $D_{t}(\mu(x),a^{2})/c^{2}$
has two extrema and those are $\pm 2a^{t}/c^{2}$. Since $t$ is odd, both extrema are
of type $(1,2,2,\ldots ,2)$. From \eqref{di2} we deduce 
\begin{equation}
g(x)=f\left(\widetilde{E}_{s}\left(\delta(x)p(x)^{2}\right)\right)
\end{equation}
with $\delta(x)$, $p(x)\in\Q[x]$ and $\deg \delta=1$, which is a decomposition that already appeared, see \eqref{dec3}.

Let now $(f_{1}(x),g_{1}(x))$ be a standard pair
of the fourth kind over $\Q$. Then
$$(f_{1}(x),g_{1}(x))=(a^{-1}D_{2}(x,a),b^{-t/2}D_{t}(x,b))$$
with even $t$. From $\eqref{felbontas1}$ we get $u(x)=(ax+2a)/c^{2}$, which together with \eqref{felbontas} implies
$$g(x)=f(\widetilde{E}_{s}(u(g_{1}(\mu(x)))))=
f\left(\widetilde{E}_{s}\left(\frac{ab^{-t/2}D_{t}(\mu(x),b)+2a}{c^{2}}\right)\right).$$
\noindent The extrema of $ab^{-t/2}D_{t}(\mu(x),b)/c^{2}$ are $\pm
2b^{t/2}ab^{-t/2}/c^{2}=\pm 2a/c^{2}$, and the extremum $-2a/c^{2}$
is of type $(2,2,\ldots ,2)$ by Theorem \ref{bdickson}.
Therefore 
\begin{equation}
g(x)=f\left(\widetilde{E}_{s}\left(p(x)^{2}\right)\right)
\end{equation} with $p(x)\in\Q[x]$.
Then the equation \eqref{EQUATION} turns into
\begin{equation}\label{eql}
f\left(\widetilde{E}_{s}\left(\left(h(x)-\frac{1}{2}\right)^{2}\right)\right)=
f\left(\widetilde{E}_{s}\left(p(y)^{2}\right)\right).
\end{equation}
If we let $r(x)$ be an integer valued polynomial which attains only positive values and $p(x)=2r(x)-1/2$ if $h(x)=2x$ and $p(x)=2r(x)+1/2$ if $h(x)=2x+1$, then $(x, y)=(r(k), k)$ are solutions of the equation \eqref{eql} for any $k\in \N$.

Since $\deg f_1=2$, the pair $(f_1(x), g_1(x))$ can not be a standard pair over $\Q$ of the fifth kind.
\end{proof}

\subsection*{Acknowledgements}
Dijana Kreso supported by the Austrian Science Fund (FWF): W1230-N13 and NAWI Graz. Csaba Rakaczki was supported, in part, by the Hungarian Academy of Sciences under OTKA Grant K75566 and the Project TAMOP-4.2.1.B-10/2/KONV-2010-0001,
which was supported by the European Union and co-financed by the European Social Fund.

\bibliographystyle{amsplain}
\bibliography{Dijana2}

\end{document}